\documentclass[12pt]{article}%
\usepackage{mitpress}
\usepackage{amsmath}
\usepackage{graphicx}
\usepackage{amsfonts}
\usepackage{amssymb}%
\providecommand{\U}[1]{\protect\rule{.1in}{.1in}}
\providecommand{\U}[1]{\protect\rule{.1in}{.1in}}
\providecommand{\U}[1]{\protect\rule{.1in}{.1in}}
\newtheorem{theorem}{Theorem}
\newtheorem{acknowledgement}[theorem]{Acknowledgement}

\newtheorem{corollary}[theorem]{Corollary}

\newtheorem{proposition}[theorem]{Proposition}
\newtheorem{remark}[theorem]{Remark}

\newenvironment{proof}[1][Proof]{\textbf{#1.} }{\ \rule{0.5em}{0.5em}}
\newdimen\dummy
\dummy=\oddsidemargin
\begin{document}

\title{A symmetry property for polyharmonic functions vanishing on equidistant hyperplanes}
\author{O. Kounchev, H. Render}
\maketitle

\begin{abstract}
Let $u\left(  t,y\right)  $ be a polyharmonic function of order $N$ defined on
the strip $\left(  a,b\right)  \times\mathbb{R}^{d}$ satisfying the growth
condition
\begin{equation}
\sup_{t\in K}\left\vert u\left(  t,y\right)  \right\vert \leq o\left(
\left\vert y\right\vert ^{\left(  1-d\right)  /2}e^{\frac{\pi}{c}\left\vert
y\right\vert }\right)  \label{eqgrowth}%
\end{equation}
for $\left\vert y\right\vert \rightarrow\infty$ and any compact subinterval
$K$ of $\left(  a,b\right)  $, and suppose that $u\left(  t,y\right)  $
vanishes on $2N-1$ equidistant hyperplanes of the form $\left\{
t_{j}\right\}  \times\mathbb{R}^{d}$ for $t_{j}=t_{0}+jc\in\left(  a,b\right)
$ and $j=-\left(  N-1\right)  ,...,N-1.$ Then it is shown that $u\left(
t,y\right)  $ is odd at $t_{0},$ i.e. that $u\left(  t_{0}+t,y\right)
=-u\left(  t_{0}-t,y\right)  $ for $y\in\mathbb{R}^{d}$. The second main
result states that $u$ is identically zero provided that $u$ satisfies
(\ref{eqgrowth}) and vanishes on $2N$ equidistant hyperplanes with distance
$c.$

\end{abstract}

\section{Introduction}

A function $f:G\rightarrow\mathbb{C}$ defined on a domain $G$ in the euclidean
space $\mathbb{R}^{d}$ is called \emph{polyharmonic of order} $N$ if $f$ is
$2N$ times continuously differentiable and
\[
\Delta^{N}f\left(  x\right)  =0
\]
for all $x\in G,$ where $\Delta=\frac{\partial^{2}}{\partial x_{1}^{2}%
}+...+\frac{\partial^{2}}{\partial x_{d}^{2}}$ is the Laplace operator and
$\Delta^{N}$ the $N$-th iterate of $\Delta.$ Polyharmonic functions play an
important role in pure and applied mathematics, and they have been studied
extensivley in \cite{ACL83}, \cite{Avan85}, \cite{Nico35}. In \cite{GGS10}
polyharmonic functions are studied in the context of boundary value problems
for partial differential operators of higher order. In applied mathematics
they are important for multivariate interpolation and spline theory
\cite{Koun00}, \cite{kounchevrenderJAT}, \cite{KoReCM}, \cite{MaNe90}, for
constructing new types of cubature formulae \cite{kounchevRenderKleinDirac},
or for constructing wavelets and subdivision schemes \cite{DKLR11},
\cite{DKLR14}. Polyharmonic functions of order $2$ are called
\emph{biharmonic} \emph{functions} and they play an eminent role in elasticity
theory, cf. the references in \cite{Gomi03}.

An important tool in the theroy of polyharmonic functions is the Almansi
decomposition which was proven in 1899 in \cite{Alma99}: for a polyharmonic
function $f$ of order $N$ defined on a star domain $G$ there exist harmonic
functions $h_{k}:G\rightarrow\mathbb{C}$ such that
\begin{equation}
f\left(  x\right)  =\sum_{k=0}^{N-1}\left\vert x\right\vert ^{2k}h_{k}\left(
x\right)  , \label{eqAlmansi}%
\end{equation}
see e.g. \cite{ACL83}. A simple consequence is the following result: if $G$ is
a ball of radius $R>0$ with center $0$ and if the polyharmonic function $f$
vanishes on the concentric spheres $\left\{  x\in\mathbb{R}^{d}:\left\vert
x\right\vert =t_{j}\right\}  $ for given radii $0<t_{1}<...<t_{N}<R$ then $f$
is identically zero. In other words: a polyharmonic function of order $N$ is
completely determined by its values on $N$ concentric spheres. This result was
generalized in \cite{Rend08} replacing concentric spheres by the boundaries of
ellipsoids in arbitrary position answering positively a question in
\cite{Hayman}.

In this paper we consider a similar question for polyharmonic functions
$u\left(  t,y\right)  $ of order $N$ defined on the strip
\[
\left(  a,b\right)  \times\mathbb{R}^{d}=\left\{  \left(  t,y\right)
:t\in\left(  a,b\right)  \text{ and }y\in\mathbb{R}^{d}\right\}  \text{.}%
\]
We ask under which conditions is it true that

\begin{itemize}
\item[(*)] a polyharmonic function of order $N$ vanishing on $2N$ hyperplanes
of the form $\left\{  t_{j}\right\}  \times\mathbb{R}^{d}$ for real numbers
$t_{1}<...<t_{2N}$ in the open interval $\left(  a,b\right)  $ is identically zero?
\end{itemize}

In passing we mention that this question arises naturally in the context of
polyharmonic interpolation for data functions given on the hyperplanes
$\left\{  t_{j}\right\}  \times\mathbb{R}^{d}$ for $j=1,..,2N$, see
\cite{KoRe14}. The case $N=1$ and $d=2$ already shows that one needs
additional assumptions for a positive answer: the harmonic function $u\left(
t,y\right)  =\sin t\cdot e^{y}$ vanishes on two parallel lines but it is not
the zero function. On the other hand, it is well known that a harmonic
function vanishing on two parallel lines is identically zero if it has
exponential growth of order less than $1.$ Thus one can expect only a positive
answer if $u\left(  t,y\right)  $ satisfies certain growth estimates. 

We shall show that statement (*) is true provided that (i) the points
$t_{1},...,t_{2N}$ are equidistant, i.e. $t_{j}=t_{1}+\left(  j-1\right)  c$
for $j=1,...,2N,$ and (ii) $u\left(  t,y\right)  $ satisfies the growth
estimate
\[
\sup_{t\in\left[  t_{1},t_{2N}\right]  }\left\vert u\left(  t,y\right)
\right\vert \leq o\left(  \left\vert y\right\vert ^{\left(  1-d\right)
/2}e^{\frac{\pi}{c}\left\vert y\right\vert }\right)
\]
for $\left\vert y\right\vert \rightarrow\infty.$ We believe that the result is
true without the assumption (i) of equidistant points but we have been unable
to provide a proof.

The proof of this result is based on a new symmetry property for polyharmonic
functions which is interesting in its own right. Let us say that a function
$u:\left(  a,b\right)  \times\mathbb{R}^{d}\rightarrow\mathbb{C}$ is \emph{odd
at} $t_{0}\in\left(  a,b\right)  $ if
\begin{equation}
u\left(  t_{0}+t,y\right)  =-u\left(  t_{0}-t,y\right)  \label{eqoddd}%
\end{equation}
for all $t\in\left(  t_{0}-c,t_{0}+c\right)  $ and $y\in\mathbb{R}^{d}$ where
$c:=\min\left\{  b-t_{0},a-t_{0}\right\}  .$ We shall use the convention that
$u\left(  t,y\right)  $ is odd if it is odd at the point $t_{0}=0.$ Note that
necessarily $u\left(  t_{0},y\right)  =0$ for all $y\in\mathbb{R}^{d}$ if $u$
is odd at $t_{0}.$ For a \emph{harmonic} function $h\left(  t,y\right)  $ the
Schwarz reflection principle shows that the converse is also true: $h$ is odd
at $t_{0}$ if and only if $h\left(  t_{0},y\right)  =0$ for all $y\in
\mathbb{R}^{d}$. Simple examples show that this is not true for biharmonic functions.

We prove the following symmetry principle: Let $u:\left(  a,b\right)  $
$\times\mathbb{R}^{d}\rightarrow\mathbb{C}$ be biharmonic and suppose that
there exists $t_{1}\in\left(  a,b\right)  $ and $c>0$ such that $a<t_{1}%
-c<t_{1}+c<b$ and
\[
u\left(  t_{1},y\right)  =0\text{ and }u\left(  t_{1}+c,y\right)  =-u\left(
t_{1}-c,y\right)
\]
for all $y\in\mathbb{R}^{d}$. If
\[
\sup_{t\in\left[  t_{1}-c,t_{1}+c\right]  }\left\vert u\left(  t,y\right)
\right\vert \leq o\left(  \left\vert y\right\vert ^{\left(  1-d\right)
/2}e^{\frac{\pi}{c}\left\vert y\right\vert }\right)  \text{ for }\left\vert
y\right\vert \rightarrow\infty
\]
then $u$ is odd at $t_{1}.$ An analogous statement holds for polyharmonic
functions of order $N$.

Let us now outline the paper: in Section $2$ we briefly investigate the
Fourier series of a polyharmonic function $h:\left(  -\pi-\delta,\pi
+\delta\right)  \times\mathbb{R}^{d}\rightarrow\mathbb{C}$ of order $N.$ If
$h\left(  t,y\right)  $ is odd at $t=0$ and $t=\pm\pi$ then it is shown that
the Fourier coefficients
\[
a_{k}\left(  y\right)  =\frac{1}{\pi}\int_{-\pi}^{\pi}h\left(  t,y\right)
\sin ktdt
\]
satisfy the equation $\left(  \Delta_{y}-k^{2}\right)  ^{N}a_{k}\left(
y\right)  =0.$ If in addition $h\left(  t,y\right)  $ satisfies the growth
assumption
\[
\sup_{t\in\left[  -\pi,\pi\right]  }\left\vert t\cdot\left(  \pi^{2}%
-t^{2}\right)  h\left(  t,y\right)  \right\vert \leq o\left(  \left\vert
y\right\vert ^{\left(  1-d\right)  /2}e^{\left\vert y\right\vert }\right)
\]
for $\left\vert y\right\vert \rightarrow\infty$ then one can prove that $h$ is
identically zero using results of Vekua, Rellich and Friedman. This result
will play a fundamental role in the next sections.

In Section $3$ the symmetry principle is proved for biharmonic functions. In
Section $4$ the general case is discussed which follows the same lines but is
technically more involved.

In Section $2$ we encountered polyharmonic functions which are odd for two
different points. In the appendix we prove by elementary means that a
polyharmonic function $u:\left(  a,b\right)  \times\mathbb{R}^{d}%
\rightarrow\mathbb{C}$ possesses a polyharmonic $2\delta$-periodic extension
$\widetilde{u}$ to $\left(  -\infty,\infty\right)  \times\mathbb{R}^{d}$
provided that $u\left(  t,y\right)  $ is odd at two different points
$t_{1}<t_{2}\in\left(  a,b\right)  $ so that $a<t_{1}-\delta$ and
$t_{2}+\delta<b$ for $\delta:=t_{2}-t_{1}$.

\section{Fourier series of polyharmonic functions on a strip}

Harmonic functions on the strip in $\mathbb{R}^{2}$ or on half spaces have
been extensively studied in \cite{Braw71}, \cite{Dura03}, \cite{Gard81},
\cite{Gold79} or \cite{Widd60}. We need extensions of these results to
polyharmonic functions of order $N$ and for general dimension $d$.

\begin{theorem}
\label{ThmFseries2}Let $\delta>0$ and assume that $h:\left(  -\pi-\delta
,\pi+\delta\right)  \times\mathbb{R}^{d}\rightarrow\mathbb{C}$ is a
polyharmonic function of order $N$ which is odd at $t=0$ and $t=\pm\pi.$ Then
the Fourier coefficients $a_{k}\left(  y\right)  $ defined in (\ref{eqFourier}%
) of the function $H_{y}\left(  t\right)  :=h\left(  t,y\right)  $ on $\left[
-\pi,\pi\right]  $ for $y\in\mathbb{R}^{d}$ satisfy the equation
\[
\left(  \Delta_{y}-k^{2}\right)  ^{N}a_{k}\left(  y\right)  =0.
\]

\end{theorem}

\begin{proof}
The function $H_{y}\left(  t\right)  :=h\left(  t,y\right)  $ is odd on
$\left[  -\pi,\pi\right]  $ for each $y\in\mathbb{R}^{d}.$ Thus the Fourier
coefficients of $H_{y}$ for the basis function $\cos kt$ vanish. Next consider
the Fourier coefficients
\begin{equation}
a_{k}\left(  y\right)  =\frac{1}{\pi}\int_{-\pi}^{\pi}h\left(  t,y\right)
\sin ktdt\text{. } \label{eqFourier}%
\end{equation}
Clearly $y\longmapsto h\left(  t,y\right)  $ is a $C^{\infty}$-function on
$\mathbb{R}^{d}$, and $t\longmapsto\frac{\partial^{s}}{\partial y_{j}^{s}%
}h\left(  t,y\right)  $ for $s=1,...N,$ is bounded on $\left[  -\pi
-\delta/2,\pi+\delta/2\right]  .$ Using Lebesgue's dominated convergence
theorem we obtain
\[
\Delta_{y}^{N}a_{k}\left(  y\right)  =\frac{1}{\pi}\int_{-\pi}^{\pi}\Delta
_{y}^{N}h\left(  t,y\right)  \cdot\sin ktdt.\text{ }%
\]

Since $h\left(  t,y\right)  $ is polyharmonic of order $N$ we know that
\[
0=\left(  \frac{d^{2}}{dt^{2}}+\Delta_{y}\right)  ^{N}h\left(  t,y\right)
=\Delta_{y}^{N}h\left(  t,y\right)  +%
%TCIMACRO{\dsum _{l=1}^{N}}%
%BeginExpansion
{\displaystyle\sum_{l=1}^{N}}
%EndExpansion
\binom{N}{l}\frac{\partial^{2l}}{\partial t^{2l}}\Delta_{y}^{N-l}h\left(
t,y\right)  .
\]
Moreover, $h\left(  t,y\right)  $ is odd at $\pm\pi$ and this obviously
implies that $\Delta_{y}^{s}h\left(  t,y\right)  $ is odd at $\pm\pi,$ so
\[
\frac{\partial^{2l}}{\partial t^{2l}}\Delta_{y}^{N-l}h\left(  \pm\pi,y\right)
=0.
\]
Repeated partial integration shows that
\begin{align*}
\Delta_{y}^{N}a_{k}\left(  y\right)   &  =-%
%TCIMACRO{\dsum _{l=1}^{N}}%
%BeginExpansion
{\displaystyle\sum_{l=1}^{N}}
%EndExpansion
\binom{N}{l}\left(  -k^{2}\right)  ^{l}\frac{1}{\pi}\int_{-\pi}^{\pi}%
\Delta_{y}^{N-l}h\left(  t,y\right)  \sin ktdt\\
&  =-%
%TCIMACRO{\dsum _{l=1}^{N}}%
%BeginExpansion
{\displaystyle\sum_{l=1}^{N}}
%EndExpansion
\binom{N}{l}\left(  -k^{2}\right)  ^{l}\Delta_{y}^{N-l}a_{k}\left(  y\right)
.
\end{align*}
From this the statement is obvious.
\end{proof}

\begin{remark}
For $N=1$ it suffices to require that $h:\left[  -\pi,\pi\right]
\times\mathbb{R}^{d}\rightarrow\mathbb{C}$ is an odd continuous function which
is harmonic on $\left(  -\pi,\pi\right)  \times\mathbb{R}^{d}$ such that
$h\left(  -\pi,y\right)  =h\left(  \pi,y\right)  =0$ for all $y\in
\mathbb{R}^{d}$. Indeed, by Schwarz's reflection principle one can extend $h$
to a harmonic function on $\left(  -\infty,\infty\right)  \times\mathbb{R}%
^{d}$.
\end{remark}

\begin{remark}
The assumption that $h\left(  t,y\right)  $ vanishes for $t=\pm\pi$ is
essential: the harmonic function $h\left(  t,y\right)  =t\cdot y$ has the
Fourier series
\begin{align*}
h\left(  t,y\right)   &  =\sum_{k=1}^{\infty}a_{k}\left(  y\right) \\
\sin kt\ \text{with }a_{k}\left(  y\right)   &  =\left(  -1\right)
^{k+1}\frac{2y}{k}.
\end{align*}
Clearly $a_{k}$ is not a solution to the equation $\Delta_{y}\left(
a_{k}\right)  =k^{2}a_{k}.$
\end{remark}

Next we need a result which goes back to I. Vekua and F. Rellich in the
1940'ies for the case $N=1$ , see \cite{Veku}, and which can be found in
\cite[p. 228]{Frie56} for general $N$ :

\begin{theorem}
\label{ThmFrie}Let $k$ be a positive number and suppose that $f:\mathbb{R}%
^{d}\rightarrow\mathbb{C}$ is a solution of the equation $\left(  \Delta
_{y}-k^{2}\right)  ^{N}f\left(  y\right)  =0.$ If
\[
\left\vert f\left(  y\right)  \right\vert =o\left(  \left\vert y\right\vert
^{\left(  1-d\right)  /2}e^{k\left\vert y\right\vert }\right)
\]
for $\left\vert y\right\vert \rightarrow\infty$ then $f\left(  y\right)  $ is
identically zero.
\end{theorem}

\begin{corollary}
\label{Cor2}Let $\delta>0$ and assume that $h:\left(  -\pi-\delta,\pi
+\delta\right)  \times\mathbb{R}^{d}\rightarrow\mathbb{C}$ is polyharmonic of
order $N$ which is odd at $t=0$ and $t=\pm\pi.$ If
\begin{equation}
\sup_{t\in\left[  -\pi,\pi\right]  }\left\vert t\cdot\left(  \pi^{2}%
-t^{2}\right)  h\left(  t,y\right)  \right\vert \leq o\left(  \left\vert
y\right\vert ^{\left(  1-d\right)  /2}e^{\left\vert y\right\vert }\right)
\label{eqbound}%
\end{equation}
for $\left\vert y\right\vert \rightarrow\infty$ then $h\left(  t,y\right)  $
is identically zero.
\end{corollary}

\begin{proof}
By Theorem \ref{ThmFseries2} it suffices to show that the Fourier coefficient
$a_{k}\left(  y\right)  $ defined in (\ref{eqFourier}) is zero. Note that
\[
a_{k}\left(  y\right)  =\frac{1}{\pi}\int_{-\pi}^{\pi}h\left(  t,y\right)
\sin ktdt=\frac{1}{\pi}\int_{-\pi}^{\pi}t\left(  \pi^{2}-t^{2}\right)  \cdot
h\left(  t,y\right)  \frac{\sin kt}{t\left(  \pi^{2}-t^{2}\right)  }dt.
\]
Since $\sin kt/t\left(  \pi^{2}-t^{2}\right)  $ is bounded on the interval
$\left[  -\pi,\pi\right]  $, it is easy to show that assumption (\ref{eqbound}%
) leads to the estimate
\[
a_{k}\left(  y\right)  =o\left(  \left\vert y\right\vert ^{\left(  1-d\right)
/2}e^{\left\vert y\right\vert }\right)  .
\]
Theorem \ref{ThmFrie} completes the proof.
\end{proof}

\section{A symmetry principle for biharmonic functions}

Biharmonic functions are difficult to investigate since they behave rather
differently from their harmonic peers. Sometimes it is instructive to consider
the univariate case: a biharmonic function $f:\left(  a,b\right)
\rightarrow\mathbb{C}$ is just a solution of the equation
\[
\frac{d^{4}}{dx^{4}}f\left(  x\right)  =0,
\]
thus $f\left(  x\right)  $ is a polynomial of degree $3.$ It is well known
that a polynomial of degree $3$ which vanishes in $4$ points is identically
zero. Moreover a polynomial $f$ satisfying $f\left(  0\right)  =0$ and
$f\left(  -c\right)  =-f\left(  c\right)  $ is odd. In this section we provide
analogs of these statements for biharmonic functions.

The following result will be our main tool:

\begin{theorem}
\label{ThmMain1}Suppose that $u\left(  t,y\right)  $ is polyharmonic of order
$N$ on $\left(  a,b\right)  \times\mathbb{R}^{d}$ and let $t_{0}\in\left(
a,b\right)  $ such that $u\left(  t_{0},y\right)  =0$ for all $y\in
\mathbb{R}^{d}$. Define $c:=\min\left\{  b-t_{0},t_{0}-a\right\}  >0.$ Then
there exists an odd polyharmonic function $H_{N-1}$ of order $N-1$ defined on
$\left(  -c,c\right)  \times\mathbb{R}^{d}$ such that
\[
u\left(  t_{0}+t,y\right)  +u\left(  t_{0}-t,y\right)  =t\cdot H_{N-1}\left(
t,y\right)  .
\]

\end{theorem}

\begin{proof}
Define a polyharmonic function $\widetilde{u}$ defined on $\left(
a-t_{0},b-t_{0}\right)  \times\mathbb{R}^{d}$ by $\widetilde{u}\left(
t,y\right)  =u\left(  t+t_{0},y\right)  $. By a theorem of Almansi (see e.g.
\cite{Alma99}, \cite{Avan85}, \cite{Frie56} or \cite{Nyst74}) there exist
harmonic functions $h_{j}:\left(  a-t_{0},b-t_{0}\right)  \times\mathbb{R}%
^{d}\rightarrow\mathbb{C}$ such that%
\[
\widetilde{u}\left(  t,y\right)  =%
%TCIMACRO{\dsum _{j=0}^{N-1}}%
%BeginExpansion
{\displaystyle\sum_{j=0}^{N-1}}
%EndExpansion
t^{j}\cdot h_{j}\left(  t,y\right)
\]
for all $\left(  t,y\right)  \in\left(  a-t_{0},b-t_{0}\right)  \times
\mathbb{R}^{d}$. Since $\widetilde{u}\left(  0,y\right)  =0$ it follows that
$h_{0}\left(  0,y\right)  =0.$ Since $h_{0}$ is harmonic it follows that
$h_{0}\left(  -t,y\right)  =-h_{0}\left(  t,y\right)  .$ Then we obtain that
\[
\widetilde{u}\left(  t,y\right)  +\widetilde{u}\left(  -t,y\right)  =t\cdot
H_{1}\left(  t,y\right)  \text{ }%
\]
for $\left\vert t\right\vert <\min\left\{  b-t_{0},t_{0}-a\right\}  $ where
\[
H_{1}\left(  t,y\right)  =%
%TCIMACRO{\dsum _{j=1}^{N-1}}%
%BeginExpansion
{\displaystyle\sum_{j=1}^{N-1}}
%EndExpansion
t^{j-1}\cdot\left(  h_{j}\left(  t,y\right)  +\left(  -1\right)  ^{j}%
h_{j}\left(  -t,y\right)  \right)
\]
is polyharmonic of order $\leq N-1.$ The function $H_{1}\left(  t,y\right)  $
is odd since $t\longmapsto u\left(  t_{0}+t,y\right)  +u\left(  t_{0}%
-t,y\right)  $ is even.
\end{proof}

\begin{theorem}
\label{ThmSymm}Let $u:\left(  a,b\right)  $ $\times\mathbb{R}^{d}%
\rightarrow\mathbb{C}$ be biharmonic and suppose that there exists $t_{1}%
\in\left(  a,b\right)  $ and $c>0$ such that $a<t_{1}-c<t_{1}+c<b$ and
\[
u\left(  t_{1},y\right)  =0\text{ and }u\left(  t_{1}+c,y\right)  =-u\left(
t_{1}-c,y\right)
\]
for all $y\in\mathbb{R}^{d}$. If
\[
\sup_{t\in\left[  t_{1}-c,t_{1}+c\right]  }\left\vert u\left(  t,y\right)
\right\vert \leq o\left(  \left\vert y\right\vert ^{\left(  1-d\right)
/2}e^{\frac{\pi}{c}\left\vert y\right\vert }\right)
\]
for $\left\vert y\right\vert \rightarrow\infty$ then $u$ is odd at $t=$
$t_{1},$ i.e. that $u\left(  t_{1}+t,y\right)  =-u\left(  t_{1}-t,y\right)  $
for all $t\in\left[  -c,c\right]  $ and all $y\in\mathbb{R}^{d}$.
\end{theorem}

\begin{proof}
Define $\widetilde{u}\left(  t,y\right)  =u\left(  t_{1}+\frac{c}{\pi}%
t,\frac{c}{\pi}y\right)  .$ Then $\widetilde{u}\left(  0,y\right)  =0$ and
$\widetilde{u}\left(  \pi,y\right)  =-\widetilde{u}\left(  -\pi,y\right)  .$
By Theorem \ref{ThmMain1} there exists an odd harmonic function $H$ such that
\[
\widetilde{u}\left(  t,y\right)  +\widetilde{u}\left(  -t,y\right)  =t\cdot
H\left(  t,y\right)  .
\]
It suffices to show that $H\left(  t,y\right)  $ is identically zero. Clearly
$H\left(  \pm\pi,y\right)  =0.$ We can now estimate
\[
\sup_{t\in\left[  -\pi,\pi\right]  }\left\vert t\cdot H\left(  t,y\right)
\right\vert \leq2\sup_{t\in\left[  -\pi,\pi\right]  }\left\vert \widetilde
{u}\left(  t,y\right)  \right\vert =2\sup_{t\in\left[  -\pi,\pi\right]
}\left\vert u\left(  t_{1}+\frac{c}{\pi}t,\frac{c}{\pi}y\right)  \right\vert
.
\]
Now the result follows from Corollary \ref{Cor2}.
\end{proof}

\begin{theorem}
\label{ThmMain2}Let $u:\left(  a,b\right)  $ $\times\mathbb{R}^{d}%
\rightarrow\mathbb{C}$ be biharmonic and suppose that there exists $t_{0}%
\in\left(  a,b\right)  $ and $c>0$ such that $a<t_{0}<t_{0}+3c<b$ and
\[
u\left(  t_{0}+jc,y\right)  =0\text{ and for }j=0,1,2,3
\]
and for all $y\in\mathbb{R}^{d}.$ If
\[
\sup_{t\in\left[  t_{0},t_{0}+3c\right]  }\left\vert u\left(  t,y\right)
\right\vert \leq o\left(  \left\vert y\right\vert ^{\left(  1-d\right)
/2}e^{\frac{\pi}{c}\left\vert y\right\vert }\right)
\]
then $u$ is identically zero.
\end{theorem}

\begin{proof}
Put $t_{j}=t_{0}+jc$ for $j=0,1,2,3.$ We apply Theorem \ref{ThmSymm} to the
point $t_{0}+c$ (we know that $u\left(  t,y\right)  $ vanishes for $t=t_{0}$
and $t_{0}+c$ and $t_{0}+2c)$ and we infer that $u$ is odd at $t_{0}+c.$
Theorem \ref{ThmSymm} applied to the point $t_{0}+2c$ shows that $u$ is odd at
$t_{0}=t_{0}+2c.$ Then it is easy to see that $u$ is as well odd at $t_{0}.$
Using a transformation of variables we can apply Corollary \ref{Cor2} and we
infer that $u$ must be identical zero.
\end{proof}

\section{A symmetry principle for polyharmonic functions}

We want to generalize Theorem \ref{ThmSymm} to the case of polyharmonic
functions $u:\left(  a,b\right)  $ $\times\mathbb{R}^{d}\rightarrow\mathbb{C}$
of order $N$ which vanish on $2N-1$ equidistant hyperplanes $\left\{
t_{j}\right\}  \times\mathbb{R}^{d}.$ By using simple transformations we may
assume that $t_{j}=j\pi$ for $j=0,\pm1,...,\pm\left(  N-1\right)  $ and
$u\left(  t,y\right)  $ is defined for all $\left\vert t\right\vert
<N-1+\delta.$

\begin{theorem}
\label{ThmPoly}Let $\delta>0$ and $c_{N}=\left(  N-1\right)  \pi+\delta$. Let
$u:\left(  -c_{N},c_{N}\right)  $ $\times\mathbb{R}^{d}\rightarrow\mathbb{C}$
be polyharmonic of order $N$ and suppose that
\[
u\left(  j\pi,y\right)  =-u\left(  -j\pi,y\right)  \quad\quad\text{ for
}j=0,1,...,N-1
\]
for all $y\in\mathbb{R}^{d}$ and suppose that
\[
\sup_{t\in\left(  -c_{N},c_{N}\right)  }\left\vert u\left(  t,y\right)
\right\vert \leq o\left(  \left\vert y\right\vert ^{\left(  1-d\right)
/2}e^{\left\vert y\right\vert }\right)
\]
for $\left\vert y\right\vert \rightarrow\infty.$ Then $u$ is odd.
\end{theorem}

\begin{proof}
We apply Theorem \ref{ThmMain1} to the point $t_{0}=0$. Then there exists a
polyharmonic odd function $H_{N-1}\left(  t,y\right)  $ of order $N-1$ defined
on $\left(  -c_{N},c_{N}\right)  \times\mathbb{R}^{d}$ such that
\[
u\left(  t,y\right)  +u\left(  -t,y\right)  =t\cdot H_{N-1}\left(  t,y\right)
\]
It follows that $H_{N-1}\left(  \pm j\pi,y\right)  =0$ for $j=0,...,N-1.$
Theorem \ref{ThmMain1} applied to the function $H_{N-1}\left(  t,y\right)  $
and the point $\pi$ shows that there exists a polyharmonic odd function
$H_{N-2}\left(  t,y\right)  $ of order $N-2$ defined on $\left(
-c_{N-1},c_{N-1}\right)  \times\mathbb{R}^{d}$ such that
\begin{equation}
H_{N-1}\left(  \pi+t,y\right)  +H_{N-1}\left(  \pi-t,y\right)  =t\cdot
H_{N-2}\left(  t,y\right)  \label{eqNEW1}%
\end{equation}
where $c_{N-1}=c_{N}-\pi.$ Clearly $H_{N-2}\left(  j\pi,y\right)  =0$ for
$j=0,...,N-2.$ Having constructed the odd polyharmonic function $H_{N-\left(
j-1\right)  }$ of order $N-\left(  j-1\right)  $ vanishing on $\left(
j\pi,y\right)  $ for $j=0,..,N-\left(  j-1\right)  $ we define the odd
polyharmonic function $H_{N-j}\left(  t,y\right)  $ of order $N-j$ by Theorem
\ref{ThmMain1} by the equation
\begin{equation}
H_{N-\left(  j-1\right)  }\left(  \pi+t,y\right)  +H_{N-\left(  j-1\right)
}\left(  \pi-t,y\right)  =t\cdot H_{N-j}\left(  t,y\right)  .
\label{eqHrecurs}%
\end{equation}
It follows that $H_{N-1}\left(  j\pi,y\right)  =0$ for $j=0,...,N-j.$ Equation
(\ref{eqrepw}) in the next theorem shows that
\begin{equation}
\sup_{t\in\left[  -\pi,\pi\right]  }\left\vert t\cdot\left(  \pi^{2}%
-t^{2}\right)  H_{N-j}\left(  t,y\right)  \right\vert =o\left(  \left\vert
y\right\vert ^{\left(  1-d\right)  /2}e^{\left\vert y\right\vert }\right)  .
\label{eqHbound}%
\end{equation}
For $j=N-1$ it follows that $H_{1}$ is a harmonic function which is odd at $0$
and $\pi$ and satisfies (\ref{eqHbound}). Corollary \ref{Cor2} shows that
$H_{1}$ is identically $0.$ It follows from (\ref{eqHrecurs}) for $j=N-1$ that
$H_{2}$ is odd at $\pi$. Corollary \ref{Cor2} and (\ref{eqHbound}) show that
$H_{2}$ is zero. Now one can proceed inductively and one obtains that
$H_{1},...,H_{N-1}$ are identically zero, which clearly implies that $u$ is odd.
\end{proof}

\begin{theorem}
Let $u\left(  t,y\right)  $ and $H_{N-j}\left(  t,y\right)  $ as in the proof
of Theorem \ref{ThmPoly} and define%
\begin{align*}
w_{N}\left(  t,y\right)   &  :=t\cdot H_{N-1}\left(  t,y\right)  =u\left(
t,y\right)  +u\left(  -t,y\right) \\
A_{j}\left(  t\right)   &  :=\left(  \pi+t\right)  \left(  2\pi+t\right)
\cdots\left(  \pi j+t\right)  .
\end{align*}
Then the following identity holds
\begin{equation}
A_{j}\left(  t\right)  A_{j}\left(  -t\right)  \cdot t\cdot H_{N-\left(
j-1\right)  }\left(  t,y\right)  =%
%TCIMACRO{\dsum _{l=0}^{j}}%
%BeginExpansion
{\displaystyle\sum_{l=0}^{j}}
%EndExpansion
p_{j,l}\left(  t\right)  w_{N}\left(  \left(  j-2l\right)  \pi+t,y\right)
\label{eqrepw}%
\end{equation}
where $p_{j,0}\left(  t\right)  =A_{j}\left(  -t\right)  $ and $p_{j,j}\left(
t\right)  =A_{j}\left(  t\right)  ,$ and $p_{j,l}\left(  t\right)  $ are
polynomials of degree $j$ for $l=1,...,j-1$ defined by
\begin{equation}
p_{j,l}\left(  t\right)  =\binom{j}{l}\prod\limits_{s=0}^{l-1}\left[  \left(
\left(  j-s\right)  \pi-t\right)  \left(  \left(  j-s\right)  \pi+t\right)
\right]  \cdot\prod\limits_{s=0}^{j-1-2l}\left(  \left(  j-l-s\right)
\pi-t\right)  . \label{eqdefpjl}%
\end{equation}

\end{theorem}

\begin{proof}
We consider at first the case $j=1.$ By (\ref{eqNEW1})
\[
tH_{N-2}\left(  t,y\right)  =H_{N-1}\left(  \pi+t,y\right)  +H_{N-1}\left(
\pi-t,y\right)  .
\]
We multiply this identity by $\pi^{2}-t^{2},$ and we verify the validity of
(\ref{eqrepw}) using $A_{1}\left(  t\right)  =\pi+t$:
\[
t\left(  \pi^{2}-t^{2}\right)  H_{N-2}\left(  t,y\right)  =\left(
\pi-t\right)  w_{N}\left(  \pi+t,y\right)  +\left(  \pi+t\right)  w_{N}\left(
\pi-t,y\right)  .
\]
For the general case multiply the equation (\ref{eqHrecurs}) with $t\cdot
A_{j}\left(  t\right)  A_{j}\left(  -t\right)  $
\begin{align}
&  t^{2}\cdot A_{j}\left(  t\right)  A_{j}\left(  -t\right)  \cdot
H_{N-j}\left(  t,y\right) \label{eqH4B}\\
&  =t\cdot A_{j}\left(  t\right)  A_{j}\left(  -t\right)  H_{N-\left(
j-1\right)  }\left(  \pi+t\right)  +t\cdot A_{j}\left(  t\right)  A_{j}\left(
-t\right)  H_{N-\left(  j-1\right)  }\left(  \pi-t\right)  . \label{eqH4C}%
\end{align}
Note that $A_{j}\left(  t\right)  =\left(  \pi+t\right)  A_{j-1}\left(
\pi+t\right)  .$ Since
\[
A_{j-1}\left(  -\left(  t+\pi\right)  \right)  =\left(  -t\right)  \left(
\pi-t\right)  \cdots\left(  \pi\left(  j-2\right)  -t\right)
\]
it follows that
\begin{equation}
\left(  -t\right)  \cdot A_{j}\left(  -t\right)  =\left(  \pi j-t\right)
\left(  \pi\left(  j-1\right)  -t\right)  \cdot A_{j-1}\left(  -\left(
t+\pi\right)  \right)  . \label{eqH4b}%
\end{equation}
If we define
\[
w_{N-j}\left(  t,y\right)  :=A_{j}\left(  t\right)  A_{j}\left(  -t\right)
\cdot t\cdot H_{N-\left(  j-1\right)  }\left(  t,y\right)
\]
then equation (\ref{eqH4B}) shows that
\begin{align}
t\cdot w_{N-j}\left(  t,y\right)   &  =-\left(  \pi j-t\right)  \left(
\pi\left(  j-1\right)  -t\right)  \cdot w_{N-\left(  j-1\right)  }\left(
\pi+t,y\right) \label{eqH5}\\
&  +\left(  \pi j+t\right)  \left(  \pi\left(  j-1\right)  +t\right)  \cdot
w_{N-\left(  j-1\right)  }\left(  t-\pi,y\right)  .\nonumber
\end{align}
where we have used that $A_{j}\left(  -t\right)  =\left(  \pi-t\right)
A_{j-1}\left(  \pi-t\right)  $ and (\ref{eqH4b}) for $-t,$ and in the last
step the fact that $w_{N-\left(  j-1\right)  }$ is even in $t.$ Unfortunately,
this recursion does not give an easy estimate of the function $w_{N-j}\left(
t,y\right)  $ due to the presence of the factor $t$ on the left hand side of
the formula.

Now we prove the existence of the representation (\ref{eqrepw}) by induction
over $j.$ For $j=1$ this has been verified and assume that it is true for
$j-1,$ namely
\[
w_{N-\left(  j-1\right)  }\left(  t,y\right)  =%
%TCIMACRO{\dsum _{l=0}^{j-1}}%
%BeginExpansion
{\displaystyle\sum_{l=0}^{j-1}}
%EndExpansion
p_{j-1,l}\left(  t\right)  w_{N}\left(  \left(  j-1-2l\right)  \pi+t,y\right)
.
\]
Thus we obtain
\begin{align*}
w_{N-\left(  j-1\right)  }\left(  t+\pi\right)   &  =%
%TCIMACRO{\dsum _{l=0}^{j-1}}%
%BeginExpansion
{\displaystyle\sum_{l=0}^{j-1}}
%EndExpansion
p_{j-1,l}\left(  t+\pi\right)  w_{N}\left(  \left(  j-2l\right)  \pi+t\right)
\\
w_{N-\left(  j-1\right)  }\left(  t-\pi\right)   &  =%
%TCIMACRO{\dsum _{l=0}^{j-1}}%
%BeginExpansion
{\displaystyle\sum_{l=0}^{j-1}}
%EndExpansion
p_{j-1,l}\left(  t-\pi\right)  w_{N}\left(  t+\left(  j-2-2l\right)
\pi\right)  .
\end{align*}
Now (\ref{eqH5}) shows that
\[
t\cdot w_{N-j}\left(  t,y\right)  =%
%TCIMACRO{\dsum _{l=0}^{j}}%
%BeginExpansion
{\displaystyle\sum_{l=0}^{j}}
%EndExpansion
\widetilde{p}_{j,l}\left(  t\right)  w_{N}\left(  \left(  j-2l\right)
\pi+t,y\right)
\]
where
\begin{align*}
\widetilde{p}_{j,0}\left(  t\right)   &  =-\left(  -\pi j-t\right)  \left(
\pi\left(  j-1\right)  -t\right)  p_{j-1,0}\left(  t+\pi\right)  \\
\widetilde{p}_{j,j}\left(  t\right)   &  =-\left(  \pi j+t\right)  \left(
\pi\left(  j-1\right)  +t\right)  p_{j-1,j-1}\left(  t-\pi\right)
\end{align*}
and for $l=1,...,j-1$
\[
\widetilde{p}_{j,l}\left(  t\right)  =-\left(  \pi j-t\right)  \left(
\pi\left(  j-1\right)  -t\right)  p_{j-1,l}\left(  t+\pi\right)  +\left(  \pi
j+t\right)  \left(  \pi\left(  j-1\right)  +t\right)  p_{j-1,l-1}\left(
t-\pi\right)  .
\]
Straightforward but tedious calculations (see Appendix 2) show that
\[
\frac{\widetilde{p}_{j,l}\left(  t\right)  }{t}=p_{j,l}\left(  t\right)  .
\]

\end{proof}

Following the proof of Theorem \ref{ThmMain2} one may derive from Theorem
\ref{ThmPoly} the main result of the paper :

\begin{theorem}
Let $u:\left(  a,b\right)  $ $\times\mathbb{R}^{d}\rightarrow\mathbb{C}$ be
polyharmonic of order $N$ and suppose that there exist $t_{0}\in\left(
a,b\right)  $ and $c>0$ such that $a<t_{0}<t_{0}+\left(  2N-1\right)  c<b$
and
\[
u\left(  t_{0}+jc,y\right)  =0\text{ \quad\quad for }j=0,1,...,2N-1
\]
and for all $y\in\mathbb{R}^{d}.$ If
\[
\sup_{t\in\left[  t_{0},t_{0}+c\left(  2N-1\right)  \right]  }\left\vert
u\left(  t,y\right)  \right\vert \leq o\left(  \left\vert y\right\vert
^{\left(  1-d\right)  /2}e^{\frac{\pi}{c}\left\vert y\right\vert }\right)
\]
for $\left\vert y\right\vert \rightarrow\infty,$ then $u$ is identically zero.
\end{theorem}

\section{Appendix 1: Periodic extensions of polyharmonic functions}

A reflection law for biharmonic functions at a hyperplane $\left\{
t_{0}\right\}  \times\mathbb{R}^{d}$ was introduced by Poritsky and extended
by Duffin (see \cite{Nyst74}). Unfortunately it is required in these results
that not only $u\left(  t_{0},y\right)  $ but also the normal derivative
$\frac{\partial}{\partial t}u\left(  t_{0},y\right)  $ vanishes for all
$y\in\mathbb{R}^{d}$. Therefore these results could not be used in our setting.

In Section 2 we used the assumption that a polyharmonic function $u\left(
t,y\right)  $ is odd at two points $t_{1}<t_{2}.$ This is a rather strong
assumption: roughly speaking, we shall prove that this implies that $u\left(
t,y\right)  $ is periodic. In order to prove this we need the following
technical result which might be part of mathematical folklore:

\begin{proposition}
\label{PropPeriod}Suppose that $u\left(  t,y\right)  $ is a polyharmonic
function on $\left(  a,b\right)  \times\mathbb{R}^{d}$ and suppose that there
exists a positive constant $c<b-a$ such that
\begin{equation}
u\left(  t,y\right)  =u\left(  t+c,y\right)  \text{ for all }y\in
\mathbb{R}^{d}, \label{eqperiod}%
\end{equation}
and for all $t\in\left(  a,b\right)  $ such that $t+c\in\left(  a,b\right)  $.
Then $u$ possesses a polyharmonic extension $\widetilde{u}$ defined on
$\left(  -\infty,\infty\right)  \times\mathbb{R}^{d}$ such that
(\ref{eqperiod}) holds for $\widetilde{u}$ and for all $t\in\left(
-\infty,\infty\right)  $ and $y\in\mathbb{R}^{d}.$
\end{proposition}

\begin{proof}
Define $V_{k}:=$ $\left(  a+kc,b+kc\right)  \times\mathbb{R}^{d}$ for each
integer $k.$ Define a polyharmonic function $u_{k}$ on $V_{k}$ by setting
$u_{k}\left(  x,y\right)  :=u\left(  x-kc,y\right)  .$ Now we define the
extension $\widetilde{u}\left(  t,y\right)  $ by setting $\widetilde{u}\left(
t,y\right)  =u_{k}\left(  t,y\right)  $ whenever $\left(  t,y\right)  \in
V_{k}.$ We have to show the correctness of the definition of $\widetilde{u}.$
Let us suppose that $\left(  t,y\right)  \in V_{k}$ and $\left(  t,y\right)
\in V_{l}$ for two different integers $k,l$. Then $a<t-kc<b$ and $a<t-lc<b.$
We may assume that $k<l.$ We have to show that
\[
u_{k}\left(  t,y\right)  =u\left(  t-kc,y\right)  =u\left(  t-lc,y\right)
=u_{l}\left(  t,y\right)  .
\]
Since $t-kc=t-lc+\left(  l-k\right)  c\in\left(  a,b\right)  $ and $l-k>0$ and
$c>0$ we infer that
\[
t_{j}:=t-lc+jc\in\left(  a,b\right)  \text{ \quad for }j=0,...,l-k.
\]
From our assumption we infer that $u\left(  t_{j},y\right)  =u\left(
t_{j+1},y\right)  $ for $j=0,...,l-k,$ and therefore $u\left(  t_{0},y\right)
=u\left(  t_{l-k},y\right)  $ which is the statement.
\end{proof}

\begin{theorem}
\label{ThmPeriod}Suppose that $u\left(  t,y\right)  $ is a polyharmonic
function on $\left(  a,b\right)  \times\mathbb{R}^{d}$ which is odd at two
different points $t_{1}<t_{2}\in\left(  a,b\right)  .$ If $a<t_{1}-\delta$ and
$t_{2}+\delta<b$ for $\delta:=t_{2}-t_{1}$ then $u$ possesses a polyharmonic
extension $\widetilde{u}$ defined on $\left(  -\infty,\infty\right)
\times\mathbb{R}^{d}$ satisfying
\begin{equation}
\widetilde{u}\left(  t+2\left(  t_{2}-t_{1}\right)  ,y\right)  =\widetilde
{u}\left(  t,y\right)  \text{ \quad\quad for all }\left(  t,y\right)
\in\left(  -\infty,\infty\right)  \times\mathbb{R}^{d}. \label{eqperioduu}%
\end{equation}

\end{theorem}

\begin{proof}
We restrict the function $u$ to the domain $\left(  t_{1}-\delta,t_{2}%
+\delta\right)  \times\mathbb{R}^{d}.$ We want to apply Proposition
\ref{PropPeriod} and we check now the validity of its assumption: Let
$t\in\left(  t_{1}-\delta,t_{2}+\delta\right)  $ be given such that
$t+2\delta\in\left(  t_{1}-\delta,t_{2}+\delta\right)  .$ We have to show that
$u\left(  t,y\right)  =u\left(  t+2\delta,y\right)  .$ Note that
$t+2\delta<t_{2}+\delta=t_{1}+2\delta,$ so $t<t_{1}$, and clearly
$t_{1}-\delta<t.$ Consider at first the point $t_{2}+s$ where $s:=t+t_{2}%
-2t_{1}.$ Note that $t_{2}+s=t+2t_{2}-2t_{1}=t+2\delta\in\left(  t_{1}%
-\delta,t_{2}+\delta\right)  \subseteq\left(  a,b\right)  .$ Further
\begin{equation}
t_{2}-s=t_{2}-t-t_{2}+2t_{1}=2t_{1}-t\in\left(  a,b\right)  \label{eqperiod1}%
\end{equation}
since $2t_{1}-t<2t_{1}+\delta-t_{1}=t_{1}+\delta=t_{2}<b$ and $2t_{1}%
-t>2t_{1}-t_{1}=t_{1}>a.$ As $u\left(  t,y\right)  $ is odd at $t_{2}$ we
infer that
\[
u\left(  t+2t_{2}-2t_{1},y\right)  =u\left(  t_{2}+s,y\right)  =-u\left(
t_{2}-s,y\right)  =-u\left(  t_{1}-\left(  t_{1}-t\right)  \right)  .
\]
Next we use that $u\left(  t,y\right)  $ is odd at $t_{1}$. Consider
$\sigma:=t_{1}-t.$ Formula (\ref{eqperiod1}) shows that $t_{1}+\sigma
=2t_{1}-t\in\left(  a,b\right)  $. As $u$ is odd at $t_{1}$ it follows that%
\[
-u\left(  t_{1}-\left(  t_{1}-t\right)  ,y\right)  =u\left(  t,y\right)  .
\]
By Proposition \ref{PropPeriod} there exists a polyharmonic extension
$\widetilde{u}$ of the function $u$ restricted to $\left(  t_{1}-\delta
,t_{2}+\delta\right)  \times\mathbb{R}^{d}.$ Since $u$ and $\widetilde{u}$ are
polyharmonic functions which agree on $\left(  t_{1}-\delta,t_{2}%
+\delta\right)  \times\mathbb{R}^{d}$ it follows that $\widetilde{u}$ is an
extension of the polyharmonic function $u.$ Similarly, the equation
(\ref{eqperioduu}) holds since it is true for the restriction of $u$ to the
open set $\left(  t_{1}-\delta,t_{2}+\delta\right)  \times\mathbb{R}^{d}.$
\end{proof}

\section{Appendix 2}

We provide here a proof for the last statement in the proof of Theorem
\ref{ThmPoly}, namely
\[
\frac{\widetilde{p}_{j,l}\left(  t\right)  }{t}=p_{j,l}\left(  t\right)  .
\]
Indeed, since $p_{j,0}\left(  t\right)  =\prod\limits_{s=0}^{j-1}\left(
\left(  j-s\right)  \pi-t\right)  $ we see that
\begin{align*}
\widetilde{p}_{j,0}\left(  t\right)   &  =-\left(  -\pi j-t\right)  \left(
\pi\left(  j-1\right)  -t\right)  p_{j-1,0}\left(  t+\pi\right) \\
&  =-\left(  \pi j-t\right)  \left(  \pi\left(  j-1\right)  -t\right)
\prod\limits_{s=0}^{j-2}\left(  \left(  j-2-s\right)  \pi-t\right) \\
&  =t\cdot\prod\limits_{s=0}^{j-1}\left(  \left(  j-s\right)  \pi-t\right)
=t\cdot p_{j,0}\left(  t\right)
\end{align*}
and similarly it follows that $\widetilde{p}_{j,j}\left(  t\right)  =t\cdot
p_{j,j}\left(  t\right)  .$ Next we compute $p_{j-1,l}\left(  t+\pi\right)  $
and $p_{j-1,l-1}\left(  t-\pi\right)  $ in order to compute $\widetilde
{p}_{j,l}\left(  t\right)  $. Note that
\[
\prod\limits_{s=0}^{l-1}\left(  \left(  j-1-s\right)  \pi-\left(
\pi+t\right)  \right)  \left(  \left(  j-1-s\right)  \pi+\pi+t\right)
=\prod\limits_{s=0}^{l-1}\left(  \left(  j-2-s\right)  \pi-t\right)  \left(
\left(  j-s\right)  \pi+t\right)
\]
and%
\[
\prod\limits_{s=0}^{l-1}\left(  \left(  j-2-s\right)  \pi-t\right)
=\prod\limits_{s=2}^{l+1}\left(  \left(  j-s\right)  \pi-t\right)  .
\]
We have
\[
\prod\limits_{s=0}^{j-2-2l}\left(  \left(  j-1-l-s\right)  \pi-\left(
\pi+t\right)  \right)  =\prod\limits_{s=0}^{j-2-2l}\left(  \left(
j-l-s-2\right)  \pi-t\right)  =\prod\limits_{s=2}^{j-2l}\left(  \left(
j-l-s\right)  \pi-t\right)
\]
It follows that
\begin{align*}
A  &  =\left(  \pi j-t\right)  \left(  \pi\left(  j-1\right)  -t\right)  \cdot
p_{j-1,l}\left(  \pi+t\right) \\
&  =\binom{j-1}{l}\prod\limits_{s=0}^{l+1}\left(  \left(  j-s\right)
\pi-t\right)  \prod\limits_{s=0}^{l-1}\left(  \left(  j-s\right)
\pi+t\right)  \prod\limits_{s=2}^{j-2l}\left(  \left(  j-l-s\right)
\pi-t\right)
\end{align*}
and therefore
\begin{align*}
&  A=\binom{j-1}{l}\prod\limits_{s=0}^{l-1}\left(  \left(  j-s\right)
\pi-t\right)  \left(  \left(  j-s\right)  \pi+t\right)  \cdot\prod
\limits_{s=0}^{j-2l}\left(  \left(  j-l-s\right)  \pi-t\right) \\
&  =p_{j,l}\left(  t\right)  \frac{j-l}{l}\left(  \left(  j-l-\left(
j-2l\right)  \right)  \pi-t\right)  =p_{j,l}\left(  t\right)  \frac{j-l}%
{l}\left(  l\pi-t\right)  .
\end{align*}
Next we consider $p_{j-1,l-1}\left(  t-\pi\right)  $ . At first we see that
\begin{align*}
&  \prod\limits_{s=0}^{l-2}\left(  \left(  j-1-s\right)  \pi-\left(
t-\pi\right)  \right)  \left(  \left(  j-1-s\right)  \pi+t-\pi\right) \\
&  =\prod\limits_{s=0}^{l-2}\left(  \left(  j-s\right)  \pi-t\right)
\prod\limits_{s=0}^{l-2}\left(  \left(  j-2-s\right)  \pi+t\right)
=\prod\limits_{s=0}^{l-2}\left(  \left(  j-s\right)  \pi-t\right)
\prod\limits_{s=2}^{l}\left(  \left(  j-s\right)  \pi+t\right)
\end{align*}
and
\[
\prod\limits_{s=0}^{j-2-2\left(  l-1\right)  }\left(  \left(  j-1-\left(
l-1\right)  -s\right)  \pi-\left(  t-\pi\right)  \right)  =\prod
\limits_{s=0}^{j-2l}\left(  \left(  j+1-l-s\right)  \pi-t\right)  .
\]
This implies
\begin{align*}
B=  &  \left(  \pi j+t\right)  \left(  \pi\left(  j-1\right)  +t\right)
p_{j-1,l-1}\left(  t-\pi\right) \\
&  =\binom{j-1}{l-1}\prod\limits_{s=0}^{l-2}\left(  \left(  j-s\right)
\pi-t\right)  \prod\limits_{s=0}^{l}\left(  \left(  j-s\right)  \pi+t\right)
\prod\limits_{s=0}^{j-2l}\left(  \left(  j+1-l-s\right)  \pi-t\right) \\
&  =\binom{j-1}{l-1}\prod\limits_{s=0}^{l-1}\left(  \left(  j-s\right)
\pi-t\right)  \prod\limits_{s=0}^{l}\left(  \left(  j-s\right)  \pi+t\right)
\prod\limits_{s=1}^{j-2l}\left(  \left(  j+1-l-s\right)  \pi-t\right) \\
&  =p_{j,l}\left(  t\right)  \frac{l}{j}\left(  \left(  j-l\right)
\pi+t\right)  .
\end{align*}
It follows that
\[
\widetilde{p}_{j,l}\left(  t\right)  =-A+B=p_{j,l}\left(  t\right)  \left[
\frac{l}{j}\left(  \left(  j-l\right)  \pi+t\right)  -\frac{j-l}{j}\left(
l\pi-t\right)  \right]  .
\]
The term in square brackets is equal to $t$ and the claim is proven.

\begin{acknowledgement}
Both authors acknowledge the partial support by the Bulgarian NSF Grant
I02/19, 2015. The first named author acknowledges partial support by the
Humboldt Foundation.
\end{acknowledgement}

\end{document}